\documentclass[11pt]{article}
\usepackage{latexsym}
\usepackage{amsfonts}
\usepackage{enumerate,graphicx}
\usepackage{verbatim}

\usepackage{amsmath, amssymb, amsthm}

\theoremstyle{definition}
\newtheorem*{acknowledgements}{Acknowledgements}

\numberwithin{equation}{section}

\newcommand\vanish[1]{}	

\newcommand\ourcomment[1]{ \textbf{[#1]} }
\newcommand\oc\ourcomment

\newtheorem{theorem}{Theorem}[section]
\newtheorem{conjecture}{Conjecture}[section]

\title{An application of the Gy\'{a}rf\'{a}s path argument}
\author{Vaidy Sivaraman}

\begin{document}
\maketitle
\begin{abstract}
We adapt the Gy\'{a}rf\'{a}s path argument to prove that $t-2$ cops can capture a robber, in at most $t-1$ moves, in the game of cops and robbers played in a graph that does not contain the $t$-vertex path as an induced subgraph. 
\end{abstract}

\maketitle

 The game of cops and robbers is played on a finite simple graph $G$. There are $k$ cops and a single robber. Each of the cops chooses a vertex to start, and then the robber chooses a vertex. And then they alternate turns starting with the cop. In the turn of cops, each cop either stays on the vertex or moves to a neighboring vertex. In the robber's turn, he stays on the same vertex or moves to a neighboring vertex. The cops win if at any point in time, one of the cops lands on the robber. The robber wins if he can avoid being captured. The cop number of $G$, denoted $c(G)$, is the minimum number of cops needed so that the cops have a winning strategy in $G$.  The question of what makes a graph to have high cop number is not clearly understood. Some fundamental results were proved in \cite{AF} and \cite{NW}. The book by Bonato and Nowakowski \cite{BN} is a fantastic source of information on the game of cops and robbers. (A quick primer on the cop number is \cite{AB}.) \\

All graphs in this article are finite and simple. For graphs $H,G$ we say that $G$ is $H$-free if $G$ does not contain $H$ as an induced subgraph. For a vertex $v$, the closed neighborhood of $v$, denoted by $N[v]$, is the set of all neighbors of $v$, including $v$ itself. For a positive integer $t$, $P_t$ will denote the path graph on $t$ vertices. The Gy\'{a}rf\'{a}s path argument was first used by Gy\'{a}rf\'{a}s \cite{AG} to prove that $P_t$-free are ``$\chi$-bounded". Since then it has become a standard proof technique in graph coloring (see \cite{SS}). Maria Chudnovsky and Paul Seymour (private communication) point out that there is a very nice adaptation of the Gy\'{a}rf\'{a}s path argument by Bousquet, Lagoutte, and Thomasse \cite{BLT} to a notoriously hard problem on induced subgraphs, the ``Erd\"{o}s-Hajnal conjecture".  The purpose of this note is show another application of it outside graph coloring to prove results quickly and efficiently. Joret, Kaminski, and Theis \cite{JKT} proved that $t-2$ cops can capture a robber in a $P_t$-free graph. Their proof proceeds in rounds and uses induction, and some essential details are missing. We prove a strengthening of their theorem. Our proof has the advantage of being conceptually simple, with the robber being captured very quickly, and for readers familiar with the Gy\'{a}rf\'{a}s path argument, can be summarized compactly as ``Place the $t-2$ cops on a Gy\'{a}rf\'{a}s path".

  \begin{theorem}
  Let $G$ be a connected $P_t$-free graph ($t \geq 3$). Then $t-2$ cops can capture the robber in at most $t-1$ moves. 
  \end{theorem}
  
  \begin{proof}
  Let $v_0 \in V(G)$. In the first move player C (the cop player) places all the $t-2$ cops in $v_0$. One of the cops is going to be stationary at $v_0$ and the other $t-3$ cops will move in the next turn. The robber will choose a vertex in $V(G) \setminus N[v_0]$. Let $v_1$ be a neighbor of $v_0$ that has a neighbor in the component $C_1$ of $G - N[v_0]$ containing the robber vertex. Now player C moves $t-3$ of his cops from $v_0$ to $v_1$. Now the robber moves to some vertex, and let $v_2$ be a neighbor of $v_1$ that has a neighbor in the component $C_2$ of $G[C_1] - N[v_1]$ containing the robber vertex. Player C moves $t-4$ of his cops from $v_1$ to $v_2$. This procedure is repeated so that at the end of $t-2$ moves, we have an induced path $v_0-v_1- \cdots -v_{t-3}$, a nested sequence of connected vertex sets $C_1 \supseteq C_2 \cdots \supseteq C_{t-3}$, where $v_i \not  \in C_i$ but $v_i$ has a neighbor in $C_i$.  We claim that that the robber can move only to a neighbor of some $v_{i}$. For, if the robber is not in the neighborhood of any $v_i$ after his move, then $v_0- \cdots v_{t-3}$ together with a shortest path from $v_{t-3}$ to the current robber vertex in $C_{t-3}$ (such a path exists because $C_{t-3}$ is connected and $v_{t-3}$ has a neighbor in $C_{t-3}$) is an induced path with at least $t$ vertices, a contradiction. The robber is caught in the next move, completing the proof.  
  \end{proof}
  
The path constructed in the proof is called a Gy\'{a}rf\'{a}s path, but in graph coloring, the component is chosen to be the one with the largest chromatic number rather than the component containing the robber. The simplicity of the argument presented immediately suggests: Can we do better if we understand more about the structure of $P_t$-free graphs? We conclude with a conjecture. 
  
   \begin{conjecture}
   Let $G$ be a connected $P_t$-free graph ($t \geq 5$). Then $t-3$ cops can capture the robber.
  \end{conjecture}
  
 The first case, namely that $2$ cops can capture a robber in a $P_5$-free graph looks particularly interesting. The absence of a structure theorem for $P_5$-free graphs, together with the class' stubborn resistance to the famous ``Erd\"{o}s-Hajnal conjecture" makes it all the more attractive.

\begin{acknowledgements}
I would like to thank Maria Chudnovsky, Alex Scott, and Paul Seymour, for teaching me the Gy\'{a}rf\'{a}s path argument, and Gwena\"{e}l Joret for discussion on the cop number and the proof in \cite{JKT}. 
\end{acknowledgements}

\end{document}